\documentclass[11pt]{article}

\usepackage{amssymb, verbatim, amsmath}
\usepackage[all]{xy}
\usepackage{amsthm}
\usepackage[utf8]{inputenc}
\usepackage{dirtytalk}

\parskip \baselineskip

\newcommand{\bc}{\begin{center}}
\newcommand{\ec}{\end{center}}
\newcommand{\be}{\begin{enumerate}}
\newcommand{\ee}{\end{enumerate}}
\newcommand{\beq}{\begin{equation}}
\newcommand{\eeq}{\end{equation}}
\newcommand{\bi}{\begin{itemize}}
\newcommand{\ei}{\end{itemize}}
\newcommand{\bd}{\begin{description}}
\newcommand{\ed}{\end{description}}
\newcommand{\ba}{\begin{array}}
\newcommand{\bea}{\begin{eqnarray*}}
\newcommand{\eea}{\end{eqnarray*}}
\newcommand{\ea}{\end{array}}
\newcommand{\bt}{\begin{tabular}}
\newcommand{\et}{\end{tabular}}

\newcommand{\bmi}{\begin{minipage}}
\newcommand{\emi}{\end{minipage}}

\newtheorem{thm}{Theorem}[section]
\newtheorem{defn}[thm]{Definition}
\newtheorem{lem}[thm]{Lemma}

\newtheorem{exa}[thm]{Example}

\newtheorem{rem}[thm]{Remark}

\newcommand{\Vector}[2]{\left(\begin{matrix} #1 \\ #2 \end{matrix} \right)}

\begin{document}

\bc {\bf\large Construction of a finite Dickson nearfield}\\[3mm]
{\sc Prudence Djagba  }

\it\small
Nelson Mandela University, Department of mathematics\\
South Africa\\
\rm e-mail: prudence@aims.ac.za
\ec
 
\normalsize

\quotation{\small {\bf Abstract:} For a Dickson pair $(q,n)$ we show that  $ \big \lbrace  \frac{q^k-1}{q-1}, 1 \leq k < n  \big  \rbrace $ forms a finite complete set of different residues modulo $n$. We  also study the construction   of a finite Dickson nearfield that arises from  Dickson pair $(q,n)$.

\normalsize

\textup{2010} \textit{MSC}: \textup{16Y30;12K05}

\section{Introduction}

The interest in nearrings and nearfields started  in $1905$ when Leonard Eugene Dickson (\cite{dickson1905finite} ) wanted to know what structure arises if one axiom in the list of axioms for skew-fields (division rings) was removed. He found that there do exist "nearfields", which fulfill all axioms for skew-fields except one distributive law. Dickson achieved this by starting with a field and changing the multiplication into a new operation. In his honor, these types of nearfields are called "Dickson nearfields". In $1966$ the first type of near-vector spaces was introduced by Beidleman \cite{beidleman1966near} which generalises the concept of a vector space to a non-linear structure and  used nearring modules  over a nearfield. Following that, in his thesis, the authors in \cite{djagbathesis} has extended the theory of Beidleman near-vector spaces. In \cite{djagbahowell18,djagba} the authors described the $R$-subgroups of finite dimensional Beidleman near-vector spaces. Zassenhauss \cite{zassenhauss1935}, Karzel and Ellers \cite{ellerskarzel1964} have solved some important problems in this area. Recently the author in \cite{djagba} has investigated on the generalized distributive set of a finite nearfield. In his thesis, the authors in \cite{djagbathesis} has extended the theory of Beidleman near-vector spaces. In \cite{djagbahowell18,djagba} the authors described the $R$-subgroups of finite dimensional
Beidleman near-vector spaces and introduced the notion of $R$-dimension, $R$-basis, seed set and seed number of an $R$-subgroup. In \cite{djagbacenter} the authors gave an alternative proof of the center of a finite Dickson nearfield.\\

A nearfield is an algebraic structure similar to a skew-field sometimes called division ring, except that it has only one of the two distributive laws.  
\begin{defn}(\cite{pilz2011near})
	A nearfield is a set $N$ together with two binary operations $ + $ (addition) and $ \cdot  $ (multiplication) satisfying the following axioms :  
	\begin{itemize}
		\item $(N,+) $ is an abelian group with the identity $0$, 
		\item $ (N,\cdot) $ is a semi-group i.e.,  $ ( a \cdot b ) \cdot c=  a \cdot (b  \cdot c ) $ for all elements $ a,b,c \in N   $  (the associative law for multiplication),
		\item $ (a+b) \cdot c = a \cdot c + b \cdot c $  for all elements $ a,b,c \in N   $ (the right distributive law),
		\item $N $ contains an element $ 1 $ such that $  1 \cdot a= a \cdot 1 =a $  for all element $ a \in N   $ (multiplicative identity), 
		\item For every non-zero element $ a $ of $ N $, there exists an element $ a^{-1} $ such that $ a \cdot a^{-1} = a^{-1}  \cdot a =1$ (multiplicative inverse).
	\end{itemize}   
	\label{th:t8}
\end{defn}
We will use $N^{\times}$ to denote $N \backslash \lbrace 0\rbrace $.
\begin{defn}
	A proper nearfield is a nearfield that is not a field. 
\end{defn}

Throughout this note  we will consider right nearfields and use $N$ to denote a nearfield.
\begin{exa}
	Fields and skew-fields are nearfields.
\end{exa}
\begin{exa} \cite{pilz2011near} 
	Consider the finite field $\big  (GF({3^2}),+ , \cdot \big )$, it is explicitly constructed in the following way
	\begin{align*}
		GF({3^2}) \cong \mathbb{Z}_3[X]/ (X^2+1).
	\end{align*}  It follows that 
	$ GF({3^2}) :=\left\lbrace 0,1,2, \beta,1+\beta,2+\beta,2\beta,1+2\beta,2+2\beta \right\rbrace $ where $\beta$ is a zero of $X^2+1 \in \mathbb{Z}_3[X] $. 
	The addition table on $ GF({3^2})$ is defined by  
	\begin{align*}
		(a+b \beta) + (c+d \beta) = (a+c) \text{mod }3 + ((b+d) \text{mod }3) \beta
	\end{align*}

	It is observed in \cite{pilz2011near} that $ N_9:=(GF({3^2}),+ , \circ )$ with a new multiplication defined by
	\begin{equation*}
		x \circ y =
		\begin{cases}
			x \cdot y \thickspace \text{if $y$}  \thickspace \mbox{is a square in}  \thickspace (GF({3^2}),+ , \cdot )  \\
			x^3 \cdot y \thickspace \mbox{otherwise}
		\end{cases}
	\end{equation*}
	is a finite proper nearfield.

	\label{th:t10}
\end{exa} 

We will see in the next section that this example of a finite nearfield is a finite Dickson nearfield. As we will prove later, it is the smallest finite proper nearfield.

\begin{defn}
	The map 
	\begin{eqnarray*}
		\psi \thickspace : F \to F  \\
		a \mapsto a^p 
	\end{eqnarray*}is called the Frobenius automorphism of $ F $.
	\label{t:7}
\end{defn}

Now, we introduce maps that are useful to define a new multiplication.
\begin{defn}(\cite{pilz2011near})
	Let $N$ be a nearfield and $\textit{Aut} (N,+,\cdot ) $ the set of all automorphisms of $N$. A map
	\begin{align*}
		\phi: \quad & N^{\times} \to  \textit{Aut} (N,+,\cdot )  \\ 
		& n \mapsto \phi_n
	\end{align*}
	is called a coupling map if for all $n,m \in N^{\times}, \phi _n \circ \phi_m= \phi _{ \phi _n (m) \cdot n}.$
\end{defn}

\begin{defn}(\cite{pilz2011near})
	Let $N$ be a nearfield and $\phi$ a coupling map on $N$. Then one defines a new binary operation on $N$ by 
	\begin{align*}
		n \circ_{\phi} m=
		\begin{cases}
			\phi _m (n) \cdot m  \thickspace \text{if $   m \neq 0$} \\ 
			0  \thickspace \text{if $m = 0$}.
		\end{cases}
	\end{align*}
	To see this, let $m,n \in N$, then if $m=0$, $n \circ_{\phi} m=0$.  If $m \neq 0, \phi_m(n) \in N$ and $ m \in N^{\times}$ so $\phi _m (n) \cdot m \in   N^{\times}$. It follows that $  n \circ_{\phi} m \in N$.  Thus $N$ is closed under the new operation.
	\label{t:5}
\end{defn}

\begin{lem}(\cite{pilz2011near})
	Let $N$ be a nearfield and $ \phi$ be a coupling map. Then the set
	\begin{align*}
		G= \lbrace  \phi _n: n \in N^{\times} \rbrace
	\end{align*} 
	is a group under composition of maps.
	\label{lemm}
\end{lem}

\begin{rem} \quad 
	\begin{itemize}
		\item  $(G, \circ)$ is a subgroup of $( \textit{Aut}(N),\circ)$.
		\item  $(G, \circ)$ is called a Dickson-group.
	\end{itemize}
\end{rem}
\begin{thm}(\cite{pilz2011near})
	Let $N$ be a nearfield and $ \phi$ be a coupling map on $N$. Then $ (N, +, \circ_{\phi})$ is again a nearfield where $ \circ_{\phi}$ is defined as in Definition \ref{t:5}.
\end{thm}

\vspace{10mm}

\section{Dickson construction}
The first finite proper nearfield was discovered by L.E Dickson \cite{dickson1905finite}.  He constructed the first example of a finite Dickson nearfield. His technique was to ''distort'' the multiplication of a finite field.
\begin{defn}(\cite{pilz2011near})
	Let $(N,+, \cdot)$ be a nearfield and $ \phi$ a coupling map on $N^{\times }$. Then $ (N,+,\circ _{ \phi})$ is called $\phi-$derivation of $(N,+, \cdot)$ and is denoted by $ N^{ \phi}$. 
	The  group $(G,\circ)$ is called the Dickson group of $ \phi$ with $G$ defined as in Lemma \ref{lemm}. $N$ is said to be a Dickson nearfield if $N$ is the $\phi-$derivation of some field $F$, i.e., $N=F^{\phi}.$ 
	\label{t:44}
\end{defn}
\begin{rem}
	Let us consider the coupling map $\phi: n \mapsto id_N$. In this case
	\begin{align*}
		n \circ_{\phi} m=
		\begin{cases}
			\phi _m (n) \cdot m = id_N(n) \cdot m =n \cdot m  \quad \text{if $ m \neq 0$} \\
			0 \quad \text{if $m = 0$}
		\end{cases}
	\end{align*}
	It is the trivial coupling map because the new operation is the same as the usual multiplication. For this coupling map we have that:
	\begin{itemize}
		\item Let $(N,+,\cdot)$ be a proper nearfield. The $\phi-$derivation of $(N,+,\cdot)$ is $(N,+,\circ _{\phi})$ i.e., $N^{\phi} =N$ is also a nearfield but not a Dickson nearfield. 
		\item Let $(F,+,\cdot)$ be a field. The $\phi-$derivation of $(F,+,\cdot)$ is $(F,+,\circ _{\phi})$ i.e., $F^{\phi} =F$. It follows that every field is a Dickson nearfield. 
	\end{itemize}
\end{rem}

We would like to construct finite Dickson nearfields. 
\begin{defn}(\cite{pilz2011near})
	A pairs of numbers $(q,n) \in \mathbb{N}^2$ is called a Dickson pair if
	\begin{itemize}
		\item $q$ is some power $p^l$ of a prime $p$,
		\item Each prime divisor of $n$ divides $q-1$,
		\item If $q \equiv 3$ $ \text{mod } 4$ implies $4$ does not divide $n$.
	\end{itemize} 
\end{defn}
\begin{exa} 
	The following pairs are Dickson numbers: $ (13,6),(7,3),(5,2),(9,2),(3,2),(4,3), $
	
	$(5,2),(5,4),(7,2),(11,2),(23,2), (59,2),(p,1)$ for $p$ prime. 
\end{exa}

\begin{lem}
	The set $ \big \lbrace  \frac{q^k-1}{q-1}, 1 \leq k < n  \big  \rbrace $ residues modulo $n$ is the set $\big \lbrace i, 0 \leq i < n \big  \rbrace$ where $(q,n)$ are Dickson pairs.
\end{lem}
\begin{proof}
	Let $i(k)=\frac{q^k-1}{q-1} $ for $k=1,\cdots,n.$
	
	We would like to show that the set $ \big  \lbrace i(1), i(2), \cdots, i(n)  \big  \rbrace $ residues modulo $n$ is the set $ \lbrace 0,1, \cdots, n-1 \rbrace$. It suffice to show that  the set $ \big \lbrace  \frac{q^k-1}{q-1}, 1 \leq k < n  \big  \rbrace $  are distinct residues modulo $n$.
	
	Suppose that 
	\begin{align}
		& \frac{q^k-1}{q-1} \equiv \frac{q^l-1}{q-1} \quad \text{mod } n \quad 1\leq k <l <n.
		\label{thd}
	\end{align}
	This implies that
	\begin{align*}
		&  1+q+\cdots + q^{k-1} \equiv 1+q+\cdots + q^{l-1} \quad \text{mod } n \\
		& \Rightarrow q^k+\cdots + q^{l-1} \equiv 0 \text{mod } n  \\
		& \Rightarrow  q^k(1+ \cdots+q^{l-k-1}) \equiv  0 \text{mod } n  \\
	\end{align*}
	By the definition of Dickson pair every prime divisor $p$ of $n$ divide $q-1$, so $p$ does not divide $q$. It follows that $\text{gcd }(q,n)=1$. Therefore
	\begin{align*}
		q^k(1+ \cdots+q^{l-k-1}) \equiv  0 \quad  \text{mod } n & \Rightarrow 1+ \cdots+q^{l-k-1}  \equiv 0 \quad \text{mod } n \\
		& \Rightarrow  \frac{q^{l-k}-1}{q-1} \equiv 0 \quad \text{mod } n. 
	\end{align*}
	Assume that $\frac{q^{t}-1}{q-1} \equiv 0 \quad \text{mod } n $ for some $ 1 \leq t < n.$ It follows that for all $i$,
	\begin{align*}
		\frac{q^{t}-1}{q-1} \equiv 0 \quad \text{mod } p_i^{\alpha_i} 
	\end{align*}
	where $n = \prod p_i^{\alpha_i}$ is the unique prime factorisation.
	
	Assume without loss of generality that $n=p^m.$
	
	We know that $ q \equiv 1 \quad \text{mod } p $. So we can write $q=1+p \epsilon$ for some $ \epsilon \in \mathbb{N}.$
	
	Assuming that $ p^m$ divides $ \frac{q^t-1}{q-1},$ we want to show that $n =p^m$ divides $t$ leads to contradiction.
	
	In fact
	\begin{align*}
		q^t = (1+ p \epsilon)^t = \sum _{k=0}^{t} \Vector {t}{k} (p \epsilon)^k.
	\end{align*}
	Hence 
	\begin{align*}
		\frac{q^t-1}{q-1} =  \sum _{k=1}^{t} \Vector {t}{k} (p \epsilon)^{k-1} = \cdots + \Vector{t}{2}p \epsilon + t.
	\end{align*}
	For instance 
	\begin{itemize}
		\item if $m=1,$  then the assumption is 
		\begin{align*}
			p \big  /  \frac{q^t-1}{q-1} \Leftrightarrow p \big / \sum _{k=1}^{t} \Vector {t}{k} (p \epsilon)^{k-1} \Leftrightarrow p /t
		\end{align*}
		leads to contradiction since $p=n> t$.
		\item if $m=2$, 
		
		\begin{align*}
			p^2 \big  /  \frac{q^t-1}{q-1} \Leftrightarrow p^2 \big / \sum _{k=1}^{t} \Vector{t}{k} (p \epsilon)^{k-1}  
			\Leftrightarrow p \big  /  \Vector{t}{2} p \epsilon +t \Rightarrow p /t 
		\end{align*}
		
		But then $\Vector{t}{2} = \frac{t(t-1)}{2}$, so $ p \big / \Vector{t}{2}$. Hence $p^2 \big / \Vector{t}{2} p \epsilon.$ Thus $p^2 /t$ leads to contradiction.
		\item By the same approach for some $m$, $p^m \big /  \frac{q^t-1}{q-1} \Rightarrow n=p^m / t $ leads to contradiction.
		
		Therefore the assumption \ref{thd} can not hold. Thus the set $ \big \lbrace  \frac{q^k-1}{q-1}, 1 \leq k < n  \big  \rbrace $  are distinct residues modulo $n$.
		
	\end{itemize}

\end{proof}

We will see in the next theorem that for each pair of Dickson numbers, we will be able to construct a finite Dickson nearfield containing $q^n$ elements. For any Dickson pair $(q,n)$, we will denote the associated Dickson nearfield by $DN(q,n)$. 
\begin{thm} (\cite{pilz2011near})
	
	For all pairs of Dickson numbers $(q,n)$, there exists some associated finite Dickson nearfields, of order $q^n$ which arise by taking the Galois Field $GF(q^n)$ and changing the multiplication such that $DN(q,n)= GF(q^n) ^ { \phi} = \big ( GF(q^n), +, \circ \big )$.
	\label{t:8} 
\end{thm}
\begin{proof} \quad 
	\begin{itemize}
		\item Let $(q,n)$ be a Dickson pair where $q=p^l$.
		\item Let $(F,+, \cdot)$ be a finite field with characteristic $p$ where $p$ is prime. There exists an integer $ln \geq 1$ such that $ \vert F \vert = p^{ln}.$ This field is called the Galois Field $F:=GF(q^n)=GF(p^{ln})$ containing $q^n$ elements. The multiplicative group $ ( F^{ \times}, \cdot)$ is cyclic (the proof can be found in \cite{lidl1994introduction}). So $ F^{ \times}$ is generated by an element denoted $g$, i.e. $  F^{ \times} = \langle g \rangle$.  Let us consider $H$, the subgroup of $ ( F^{ \times}, \cdot)$ generated by $g^n$, i.e., $ H= \langle g^n \rangle$.  So $  F^{ \times} / H $ is the group of all right cosets of $H$. Each coset is of the form $Hg^j = \lbrace hg^j, \forall g^j \in F^{\times} \rbrace$ where $j =0,\ldots,n-1$. Since $H$ is a subgroup of $F^{\times}$, the number of right cosets of $H$ in $F^{\times}$ is the index $ (F^{\times}:H)$ of $H$ in $F^{\times}$. Since $F^{\times}$ is finite $ (F^{\times}:H)$ is finite and by Lagrange's Theorem $  (F^{\times}:H) = \vert F^{\times}/H \vert =n= \frac{ \vert F^{\times} \vert }{\vert H \vert} $.
		Thus   
		\begin{align*}
			F^{ \times} / H =\big \lbrace  Hg^{j} : 0 \leq j \leq n-1 \big \rbrace= \big \lbrace  Hg^{0},Hg^{1},\ldots, Hg^{n-1} \big \rbrace.  
		\end{align*}
		Let $i(k) = \frac{q^k-1}{q-1}$ for $k=1,\ldots,n$. It can be shown that the set  $\big \lbrace i(1),i(2), \ldots, i(n) \big \rbrace$ forms a complete set of the powers of the coset representatives because the set $\big \lbrace i(1),i(2), \ldots, i(n) \big \rbrace$ of residues modulo $n$ give the set $ \big \lbrace 0,1,\ldots,n-1\big \rbrace $. Therefore $  F^{ \times} / H $ can also be represented as follows
		\begin{align*}
			F^{ \times} / H  & = \big \lbrace  Hg^{i(1)},Hg^{i(2)},\ldots, Hg^{i(n)} \big \rbrace  = \big \lbrace  Hg^{\frac{q^1-1}{q-1}},Hg^{\frac{q^2-1}{q-1}},\ldots, Hg^{\frac{q^n-1}{q-1}} \big \rbrace.
		\end{align*}
		\item Now let us consider 
		\begin{align*}
			\alpha : \quad & F \to F  \\
			& f \mapsto f^q
		\end{align*}
		which is a power of the Frobenius automorphism, i.e., $ \alpha=\psi ^l$ (by Definition \ref{t:7}).
		\item The map 
		\begin{align*}
			\lambda : \quad & F^{ \times} / H  \to Aut(F,+,\cdot)  \\
			& Hg^{\frac{q^k-1}{q-1}} \mapsto \alpha^k
		\end{align*}
		is well-defined: suppose $Hg^{\frac{q^{k_1}-1}{q-1}}, Hg^{\frac{q^{k_2}-1}{q-1}} \in F^{\times}/H   $ such that $ Hg^{\frac{q^{k_1}-1}{q-1}}=Hg^{\frac{q^{k_2}-1}{q-1}}$. Then 
		\begin{align*}
			& Hg^{\frac{q^{k_1}-1}{q-1}}=Hg^{\frac{q^{k_2}-1}{q-1}} 
			\Rightarrow g^{\frac{q^{k_1}-1}{q-1}} =g^{\frac{q^{k_2}-1}{q-1}} 
			\Rightarrow \frac{q^{k_1}-1}{q-1} =\frac{q^{k_2}-1}{q-1}  
			\Rightarrow k_1=k_2  \\
			& \Rightarrow \alpha^{k_1}=\alpha^{k_2} 
			\Rightarrow \lambda \big (Hg^{\frac{q^{k_1}-1}{q-1}} \big ) = \lambda \big (Hg^{\frac{q^{k_2}-1}{q-1}} \big ). 
		\end{align*}
		\item The map  \begin{align*}
			\pi  : \thickspace & F^{ \times}  \to F^{ \times} / H  \\
			&  f \mapsto Hg^{\frac{q^k-1}{q-1}}
		\end{align*} is a canonical bijection which satisfies the homomorphism property. So $\pi$ is a canonical bijection.
		\item The composition map is defined as
		\begin{align*}
			\phi = \lambda \circ \pi  : \thickspace & F^{ \times}  \to  \textit{Aut}(F,+,\cdot) \\
			& f \mapsto \alpha^k \thickspace \mbox{for} \thickspace f \in Hg^{\frac{q^k-1}{q-1}}
		\end{align*}  which is a coupling map on $ F^ {\times}.$
		We need to show that $DN(q,n)=F^{\phi}$  i.e., $ \phi _a \circ \phi _b = \phi _{ \phi _a (b) a}$ for all $ a,b \in F^{\times}.$ 
		
		Since $F^{\times}/H$ can be presented as $  F^{ \times} / H   = \big \lbrace  Hg^{i(1)},Hg^{i(2)},\ldots, Hg^{i(n)} \big \rbrace$ then 
		\begin{align*}
			F^{\times}= Hg^{i(1)} \cup Hg^{i(2)} \cup \cdots \cup Hg^{i(n)}.
		\end{align*}
		Therefore the elements of $F^{\times}$ can be written as $g^{\frac{q^k-1}{q-1}+n\delta}$ for $ \delta \in \mathbb{N}$ and $1 \leq k \leq n$. It follows that if $a=g^{i(k_1)+n \delta_1}$ and $b=g^{i(k_2)+n \delta_2}$, then 
		$ \pi (a) = Hg^{\frac{q^{k_1}-1}{q-1}}, \thickspace \pi (b) = Hg^{\frac{q^{k_2}-1}{q-1}}.$ So $ \phi_a= ( \lambda \circ \phi) (a) = \alpha^{k_1} $ and  $ \phi_b=( \lambda \circ \phi) (b)= \alpha^{k_2} $. It follows that $ \phi _a \circ \phi _b=\alpha^{k_1} \circ \alpha^{k_2} =\alpha^{k_1+ k_2}. $ 
		
		Also $  \phi _a (b) a = \alpha^{k_1} (b) \cdot a = b^{q^{k_1}} a = g^{ \big ( \frac{q^{k_2}-1}{q-1} + n \delta_2 \big )q^{k_1} }g^{\frac{q^{k_1}-1}{q-1}+ n \delta_1}= $
		
		 $g^{\frac{q^{k_1+k_2}-q^k_1+q^{k_1}-1}{q-1} +n \delta_2 q^{k_1}+n \delta_1}=g^{\frac{q^{k_1+k_2}-1}{q-1} +n ( \delta_1+ q^{k_1}\delta_2)}$. It follows that $ \phi _{ \phi _a (b) a} = \alpha^{k_1+ k_2} $. Thus $\phi _a \circ \phi _b = \phi _{ \phi _a (b) a}$ . 
		
		Thus if we consider the field $F:=(GF(q^n),+,\cdot )$ and the coupling map $ \phi$ such that $DN(q,n) = F^{\phi} = (GF(q^n),+,\circ _{\phi})$ (as a $\phi-$derivation of the finite field $F$). Then by Definition \ref{t:44} $DN(q,n) $ is a Dickson nearfield containing $q^n$ elements.
	\end{itemize} 
\end{proof}
\begin{lem}
	For all Dickson pair $(q,n)$ where $n \neq 1$, any Dickson nearfields constructed by the Galois Field $GF(q^n)$ are proper finite nearfields.
\end{lem}

\begin{proof}
	From finite Dickson construction 
	\begin{align*}
		DN(q,n):= GF(q^n)^{\phi} = \big (GF(q^n), +, \circ \big ).
	\end{align*}
	We would like to show that $\big (GF(q^n), +, \circ \big )$ is not fields i.e., there exists $a,b \in (GF(q^n) $ such that $a \circ b \neq b \circ a.$
	The coupling map is 
	\begin{align*}
		\phi = \lambda \circ \pi  : \quad & F^{ \times}  \to  \textit{Aut}(F,+,\cdot) \\
		& f \mapsto \alpha^k \quad \mbox{for} \quad k=1,\cdots,n. \\
		\Leftrightarrow \phi  : f \mapsto \begin{cases}
			\alpha  \quad   \text{if $f \in Hg^{\frac{q-1}{q-1}} $} \\
			\alpha^2    \quad  \text{if $ f \in Hg^{\frac{q^2-1}{q-1}}$} \\ 
			\vdots \quad  \quad \vdots \\ 
			\alpha^n  \quad   \text{if $f \in Hg^{\frac{q^n-1}{q-1}}$}.
		\end{cases} 
	\end{align*} 
	For $a,b \in (GF(q^n)$
	\begin{align*}
		a \circ _{\phi} b= 
		\begin{cases}
			\phi_a(n) \cdot b \quad \text{if $ b \neq 0 $} \\ 
			0 \quad  \text{if $b=0$}
		\end{cases} =
		\begin{cases}
			\alpha (a) \cdot b  \quad  \text{if $ b \in Hg^{\frac{q-1}{q-1}} $} \\
			\alpha^2 (a) \cdot b \quad    \text{if $ b \in Hg^{\frac{q^2-1}{q-1}} $} \\ 
			\vdots \quad \quad \vdots \\
			\alpha^n (a) \cdot b \quad  \text{if $ b \in Hg^{\frac{q^n-1}{q-1}} $} 
		\end{cases} = 
		\begin{cases}
			a^q	 \cdot b   \quad  \text{if $ b \in Hg^{\frac{q-1}{q-1}} $} \\
			a^{q^2} \cdot b \quad    \text{if $ b \in Hg^{\frac{q^2-1}{q-1}} $} \\ 
			\vdots \quad \quad \vdots \\
			a^{q^n} \cdot b \quad  \text{if $ b \in Hg^{\frac{q^n-1}{q-1}} $} 
		\end{cases} 
	\end{align*}
	Let $a=g^n \in Hg^{\frac{q^n-1}{q-1}}$ and $ b=g Hg^{\frac{q^1-1}{q-1}}$
	
	We have 
	\begin{align*}
		g^n \circ g & = \alpha^1(g^n)g \\
		& = (g^n)^qg \\
		& = g^{nq+1}.
	\end{align*}
	Also 
	\begin{align*}
		g \circ g^n & =\alpha ^n (g) g^n \\
		& =g^{n+1} \quad \mbox{because} \quad \alpha^n =id.
	\end{align*}
	Assume that $g^{nq+1}= g^{n+1},$ then $ g^{n(q-1)}=1.$
	
	But since $ F^{\times} = \langle g \rangle$, then $ \texttt{ord }(g) =q^n-1.$ 
	
	It follows that if $g^t=1 \Rightarrow q^n-1 / t$.
	
	Moreover since 
	\begin{align*}
		g^{n(q-1)}=1 & \Rightarrow q^n-1 / n(q-1) \\
		& \Rightarrow 1+q+\cdots+q^{n-1} /n
	\end{align*}
	But $q=p^l > 1$ so $1+q+\cdots+q^{n-1} > n$. It follows that $1+q+\cdots+q^{n-1}$ does not divides $n.$
	
	Thus $g^{n(q-1)} \neq 1.$ This means that $g^n \circ g \neq g \circ g^n$.
	
	There exists  $a=g^n \in Hg^{\frac{q^n-1}{q-1}}$ and $ b=g Hg^{\frac{q^1-1}{q-1}}$ such that $n \circ b \neq a \circ b.$ 	Thus the finite Dickson nearfields associated to the pair $(q,n)$ where $q \neq 1$ are proper finite nearfields (not fields).
\end{proof}

\begin{thm} \cite{dickson1905finite}
	By taking all pairs of Dickson numbers, all finite Dickson nearfields arise in the way described in Theorem \ref{t:8}. 
\end{thm}
\begin{proof}
	See \cite{dickson1905finite} for more details.
\end{proof}
\begin{rem} \quad
	\begin{itemize}
		\item For $n=1$, we have that $ F^{\times} = \langle g \rangle$ and $H = \langle g \rangle$. It follows that $ F^{\times} /H = \big \lbrace H \big \rbrace$ and then the coupling map is the identity since $\phi : f \mapsto \alpha= f^q=f$. Thus all finite fields  arise from the Dickson pair $(q,1)$. 
		\item Note that for all pairs $(q,n)$ where $n=1$ there exists a unique finite Dickson nearfield of order $q=p^l$ which are also finite fields.
		\item For $n \neq 1$ there exists finite proper nearfields of order $q^n.$ In this case, they are not necessarily unique finite Dickson nearfields of order $q^n.$ The number of non-isomorphic Dickson nearfields derived by this construction  (for different choices of $g$) is given by $ \frac{ \varphi (n)}{i}$, where $ \varphi $ is the Euler-function and $i$ is the order of $p (\text{mod } n ).$ 
		\item In general Dickson nearfields can be either finite or infinite. 
	\end{itemize}
\end{rem}

\section{Concluding comments}
As differences, for a finite field up to
isomorphism, there exists a unique finite field of order $p^n$ , but for a finite Dickson nearfield that arises
from the pair $(q, n)$, there does not exist a unique finite Dickson nearfield. The multiplicative group of
a finite field is cyclic but the multiplicative group of a Dickson nearfield is metacyclic.

\end{document}